\documentclass[11pt]{article}

\usepackage{epsfig}
\usepackage{graphicx}
\usepackage{amsbsy}
\usepackage{amsmath}
\usepackage{amsfonts}
\usepackage{amssymb}
\usepackage{textcomp}
\usepackage{hyperref}
\usepackage{aliascnt}

\newcommand{\mcm}[3]{\newcommand{#1}[#2]{{\ensuremath{#3}}}} 

\mcm{\tuple}{1}{\langle #1 \rangle}
\mcm{\name}{1}{\ulcorner #1 \urcorner}
\mcm{\Nbb}{0}{\mathbb{N}}
\mcm{\Zbb}{0}{\mathbb{Z}}
\mcm{\Rbb}{0}{\mathbb{R}}
\mcm{\Cbb}{0}{\mathbb{C}}
\mcm{\Qbb}{0}{\mathbb{Q}}
\mcm{\Bcal}{0}{\cal B}
\mcm{\Ccal}{0}{\cal C}
\mcm{\Dcal}{0}{\cal D}
\mcm{\Ecal}{0}{\cal E}
\mcm{\Fcal}{0}{\cal F}
\mcm{\Gcal}{0}{\cal G}
\mcm{\Hcal}{0}{\cal H}
\mcm{\Ical}{0}{\cal I}
\mcm{\Jcal}{0}{\cal J}
\mcm{\Kcal}{0}{\cal K}
\mcm{\Lcal}{0}{\cal L}
\mcm{\Mcal}{0}{\cal M}
\mcm{\Ncal}{0}{\cal N}
\mcm{\Ocal}{0}{{\cal O}}
\mcm{\Pcal}{0}{{\cal P}}
\mcm{\Qcal}{0}{{\cal Q}}
\mcm{\Rcal}{0}{{\cal R}}
\mcm{\Scal}{0}{{\cal S}}
\mcm{\Tcal}{0}{{\cal T}}
\mcm{\Ucal}{0}{{\cal U}}
\mcm{\Vcal}{0}{{\cal V}}
\mcm{\Xcal}{0}{{\cal X}}
\mcm{\Ycal}{0}{{\cal Y}}
\mcm{\Mfrak}{0}{\mathfrak M}

\DeclareMathOperator{\Cl}{Cl}
\mcm{\restric}{0}{\upharpoonright}
\mcm{\upset}{0}{\uparrow}
\mcm{\onto}{0}{\twoheadrightarrow}
\mcm{\smallNbb}{0}{{\small \mathbb{N}}}
\DeclareMathOperator{\preop}{op}
\mcm{\op}{0}{^{\preop}}

\newcommand{\se}{\subseteq}
{\begin{array}{c}
\setlength{\unitlength}{1em}}%
{\end{array}}

\usepackage{amsthm}

\newcommand{\theoremize}[2]{\newaliascnt{#1}{thm} \newtheorem{#1}[#1]{#2} \aliascntresetthe{#1}}

\theoremstyle{plain}
\newtheorem{thm}{Theorem}[section]
\theoremize{lem}{Lemma}
\theoremize{sublem}{Sublemma}
\theoremize{claim}{Claim}
\theoremize{rem}{Remark}
\theoremize{prop}{Proposition}
\theoremize{cor}{Corollary}
\theoremize{que}{Question}
\theoremize{oque}{Open Question}
\theoremize{con}{Conjecture}

\theoremstyle{definition}
\theoremize{dfn}{Definition}
\theoremize{eg}{Example}
\theoremize{exercise}{Exercise}
\theoremstyle{plain}

\usepackage{verbatim}
\usepackage{enumerate}
\usepackage[all]{xy}

\usepackage{subfig}

\title{\scshape Even an infinite bureaucracy eventually makes a decision}

\author{Johannes Carmesin}

\newcommand{\sm}{\setminus}

\begin{document}
 
\maketitle

\begin{abstract}
We show that the fact that a political decision filtered through a finite tree of committees
gives a determined answer generalises in some sense to infinite trees.
This implies a new special case of the Matroid Intersection Conjecture.
\end{abstract}

\section{Introduction}

To understand our main result, it might help to have the following real world situation in mind.
Suppose you have two political parties, the blue party $B$ and the red party $R$.
Both parties want to ratify a certain program.
For this they have to get the majority in a certain committee but the members of this committee are elected in other committees whose members are elected in still other committees and so on. If a vote is tide no member is sent.  
This is modelled by a rooted tree $T$
that is directed towards the root. The leaves\footnote{
In our context, a \emph{leaf} is a vertex that has no incoming edges.} are the members of $B$ and $R$ and the other nodes are the committees. 
The members that are allowed to vote in a certain committee come from the upward neighbours of that node. In particular, the final vote takes place in the root-committee. 
It is clear that if $T$ is finite then either
$B$ can ratify their program or $R$ can prevent them from doing that. 
The objects we are interested in in this paper can be thought of as infinite analogues of witnesses to these
two possibilities.

Given an edge set $X$ and a vertex $v$, the \emph{accumulation $A(v,X)$} of $X$ at $v$
is the difference between the number of edges in $X$ pointing to $v$ and the 
number of edges in $X$ pointing away from $v$.
Given two edge sets $X$ and $Y$ and a vertex $v$, then the \emph{accumulation }
from $X$ to $Y$ at $v$ is $A(v,X,Y)=A(v,X)-A(v,Y)$.
For an edge set $X$, we denote by $V(X)$ the set of those vertices that are incident with an edge of $X$.

The \emph{blue flow $b(X)$} to an edge set $X$ is the
union of the edge sets of those paths starting at a blue leaf 
all of whose interior vertices are not incident with an edge from $X$. 
A \emph{red blockage} is a rayless edge set $X$ such that
$V(X)$ does not meet $B$ and for every $v\in (V(X)+r)\sm R$ we have $A(v,X,b(X))\geq 0$.
A red blockage $X$ is \emph{strong} if $A(r,X,b(X))\geq 1$.
Similarly, one defines red flows $r(X)$ and blue blockages, and strong blue blockages.
In the real world situation, a red blockage witnesses that the red party can prevent the blue party from ratifying their program. Whereas, a strong blue blockage witnesses that the blue party can ratify their program.
The raylessness corresponds to the idea that the final decision should not rely on a 
chain of decisions stretching back to infinity with no `first cause'. 
The main result of this paper is the following.

\begin{thm}\label{blockage}
 Let $T$ be a locally finite rooted tree with disjoint sets $B$ and $R$ of leaves.
Then either there is a strong blue blockage or a red blockage.
\end{thm}

If we leave out the assumption that blockages are rayless, then this weaker version can easily be proved by compactness. 
But it appears that no such compactness argument will work in our situation and we have to introduce some new ideas.

Unlike for finite trees, the ``or'' in \autoref{blockage} is not exclusive in general, see \autoref{fig:blockage}.

\begin{figure} [htpb]   
	\begin{center}
		\includegraphics[width=7cm]{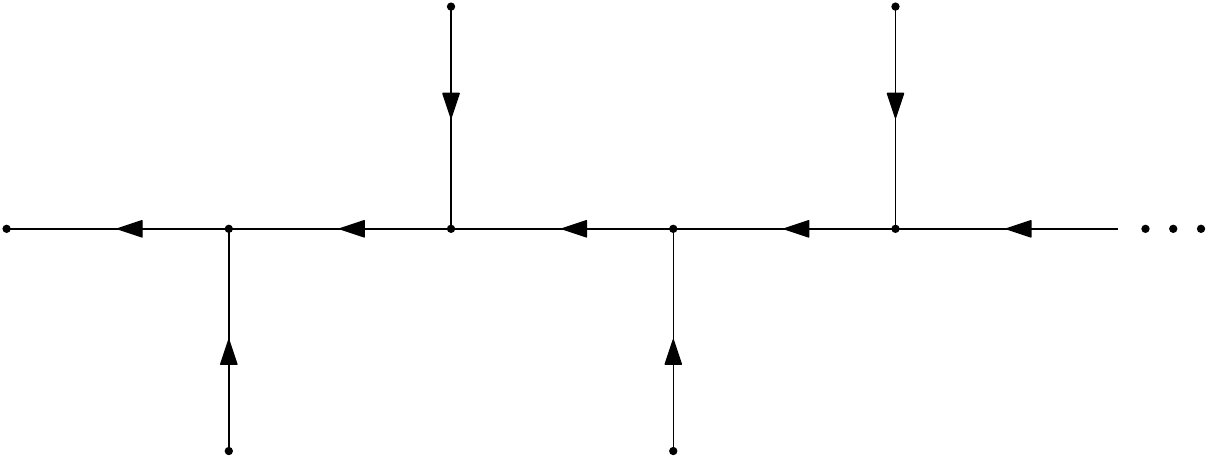}
		\caption{Let $B$ be set of bottom vertices of this tree $T$ and $R$ be the set set of top vertices. Then $(T,B,R)$ has both a strong blue blockage and a red blockage.}
		\label{fig:blockage}
	\end{center}
\end{figure}

We can use \autoref{blockage} to make some progress on the Matroid Intersection Conjecture, which says that any two infinite matroids $M$ and $N$ on a 
common ground set $E$ have a common independent set $I$ admitting a partition $I
= J_M \cup J_N$ such that $\Cl_{M}(J_M) \cup \Cl_{N}(J_N) = E$. This conjecture is known to imply the Infinite Menger theorem \cite{union2}, which had been conjectured by 
Erd\H{o}s and had been open for about 50 years until it was finally proved by Aharoni and Berger \cite{AharoniBerger}.
The Matroid Intersection Conjecture also implies tree-packing and tree-covering theorems for infinite graphs \cite{BC:PC}. 
For an introduction to this conjecture see \cite{BC:PC}. 
We use \autoref{blockage} to prove the Matroid Intersection Conjecture for two matroids $M$ and $N$ that can be decomposed into large finite uniform matroids along a tree of 2-separations, see \autoref{case_int} for details.

The paper is organised as follows.
After proving \autoref{blockage} in \autoref{sec:proof},
we deduce a special case of Matroid Intersection from it in \autoref{sec:int}.

\section{Proof of \autoref{blockage}.}\label{sec:proof}

Throughout, notation and terminology for graphs are those of~\cite{DiestelBook10}, and for matroids those of~\cite{Oxley,matroid_axioms}. Throughout this section, we fix a tree $T$ directed to its root $r$ with edge set $E$ and disjoint sets $B$ and $R$ of leaves of $T$.
Edges are ordered pairs $st$ pointing from $s$ to $t$.
Given an edge set $X$, the set $ter(X)$  of \emph{terminal vertices} of $X$ consists of those vertices $s$ such that there is no vertex $t$ such that $st\in X$.

\begin{dfn}\label{def:overflow}
The \emph{blue overflow $\bar b(X)$} of $X\se E$ is defined by the following recursive construction. 
We start by taking $Y_0$ to be the set of all edges in $E\sm X$ whose starting vertex is in $B$.
Assume that for all $\gamma<\alpha$ the set $Y_\gamma$ is defined.
First assume that $\alpha=\beta+1$ is a successor. If for all $st\in E\sm (X\cup Y_\beta)$, we have 
$A(s,Y_\beta,X)\leq 0$, we stop and let $\bar b(X)=Y_\beta$. 
Otherwise we can pick some $st\in E\sm (X\cup Y_\beta)$ with $A(s,Y_\beta,X)\geq 1$ and let
$Y_\alpha=Y_\beta+st$.
If $Y_\alpha$ is a limit, we simply let $Y_\alpha=\bigcup_{\beta<\alpha} Y_\beta$.
\end{dfn}

Clearly, this construction terminates.
As every vertex has at most one outgoing edge, we get the following.

\begin{rem}\label{flow}
$\bar b(X)$ does not depend on the choices made during the construction. Moreover,
$A(v,\bar  b(X),X)\geq 0$ for every $v\in V(\bar  b(X))\sm (B\cup ter(\bar  b(X)))$.
\qed
\end{rem}

\begin{rem}\label{smallest}
$\bar b(X)$ is the smallest set of edges containing all edges starting at blue vertices such that 
for all $st\in E\sm (X\cup \bar b(X))$ we have 
$A(s,\bar b(X),X)\leq 0$.
\qed
\end{rem}

Similar to $\bar b(X)$, one defines the red overflow $\bar r(X)$. 
If $X$ is a red blockage, then  $\bar b(X)=b(X)$.
In general, the overflow $\bar b(X)$ includes the flow $b(X)$.

\begin{lem}\label{rem1}
If $A(v,X,\bar r(X))\geq 0$ for all $v\in (V(X)\sm (B\cup ter(X)))+r$, then there is some
blue blockage $X'$ included in $X$ such that $A(r,X',r(X'))=A(r,X,\bar r(X))$.
\end{lem}

\begin{proof}
If $t\neq r$, then we say that the edge $st\in X$ is \emph{bad} if 
the unique edge starting at $t$ is in $\bar r(X)$.
We obtain $X'$ from $X$ by removing all bad edges.
Using the definition of $\bar r$, it is straightforward to check that 
$\bar r(X)=r(X')$.

As any $st\in X$ with  $A(t,X,\bar r(X))\leq -1$ is bad,
we have for all $uw\in X'$ that $A(w,X',r(X'))= A(w,X,\bar r(X))\geq 0$.
Moreover $A(r,X',r(X'))=A(r,X,\bar r(X))$ since $A(r,X,\bar r(X))\geq 0$.
Thus $X'$ is a blue blockage.
\end{proof}

\begin{lem}\label{calc1}
If $X\se Y$, then $\bar b(Y)\se \bar b(X)$ and $\bar r(Y)\se \bar r(X)$.
\end{lem}

\begin {proof}
By symmetry, it suffices to prove  $\bar b(Y)\se \bar b(X)$.
Note that $\bar b(X)$ contains all the edges starting at blue vertices.
As $X\se Y$, 
for all $st\in E\sm (Y\cup \bar b(X))$ we have 
$A(s,\bar b(X),Y)\leq 0$.
Thus $\bar b(Y)\se \bar b(X)$ by \autoref{smallest}.

\end {proof}

\begin{lem}\label{calc6}
 There are sets $X$ and $Y$ such that $\bar r(X)=Y$ and $\bar b(Y)=X$.
 \end{lem}

\begin{proof}
We define $f$ via $f(X)=\bar b (\bar r(X))$.
By \autoref{calc1}, if $X\se Y$, then $f(X)\se f(Y)$.
It suffices to construct some set $X$ such that $f(X)=X$.

For this, we consider the ordinal indexed sequence 
$f^{\alpha}$, where $f^0=\emptyset$, and if $\alpha=\beta+1$ is a successor, we just let $f^\alpha=f(f^\beta)$. 
If $\alpha$ is a limit, we let $f^\alpha=\bigcup_{\beta<\alpha} f^\beta$. 

Next we prove by induction over $\alpha$ that $f^\beta\se f^\alpha$
for all $\beta< \alpha$. 
If $\alpha$ is a limit, this is true by definition. 
So we may assume that there is some $\gamma$ with $\alpha=\gamma+1$.
By the induction hypothesis it suffices to show that $f^\gamma\se f^\alpha$.
By the induction hypothesis, $f^{\delta}\se f^\gamma$ for all $\delta<\gamma$.
So $f^{\delta+1}\se f^\alpha$.
If $\gamma$ is a successor, this immediately gives that $f^\gamma\se f^\alpha$. 
Otherwise $f^\gamma =\bigcup_{\delta<\gamma} f^{\delta+1}$.
So we also get that $f^\gamma\se f^\alpha$ in this case.
Thus $f^\beta\se f^\alpha$ for all $\beta< \alpha$.

Thus there has to be an ordinal $\gamma$ such that
$f^{\gamma}=f^{\gamma+1}$.
Hence we can pick $X=f^{\gamma}$, which completes the proof.
\end{proof}

A \emph{leafless forest} is an edge set $S$ such that the subforest $(V(S),S)$ of $T$ does not have a leaf.
Given an edge set $X$ and $c:V(T)\to \Nbb$,
then $S\se X$ is  \emph{illegal} for $c$ if $S$ is a leafless forest and 
$A(s,X\sm S)\leq c(s)$
for all $st\in S$.
A set $X$ is \emph{legal} for $c$ if no any nonempty 
subset of $X$ is illegal for $Y$.
For example, any rayless set is legal.

\begin{lem}\label{legal}
$\bar  b(Y)$ is legal for the function $v\mapsto A(v,Y)$. 
\end{lem}

\begin{proof}
Suppose for a contradiction that there is some nonempty $S\se \bar b(Y)$ that is illegal.
Amongst all $e\in S$, we pick $st$ such that it was first added to $\bar b(Y)$ in the recursive construction of \autoref{def:overflow}. Let $Y_{\alpha+1}$ be the first set in the construction containing $st$.
Then $A(s,\bar b(Y)\sm S,Y)\geq A(s,Y_\alpha,Y)\geq 1$. Thus $A(s,\bar b(Y)\sm S)\geq 1+c(s)$. This contradicts that $S$ is illegal, which completes the proof.
\end{proof}

\begin{lem}\label{cutoff_trees}
Let $X$ be legal for $c:V\to \Nbb$.
Then there is some rayless $X'\se X$ such that
\begin{enumerate}
 \item $A(v,X')\geq c(v)$ for all $v$ with $A(v,X)\geq c(v)$; 
\item $A(v,X')=A(v,X)$ for all $v\in ter(X)$.
\end{enumerate}
\end{lem}

\begin{proof}
We shall construct $X'$ as a limit of a nested decreasing ordinal indexed sequence $(X_{\alpha})$ with the following properties.
\begin{enumerate}
 \item $A(v,X_\alpha)\geq c(v)$ for all $v$ with $A(v,X)\geq c(v)$;
\item $A(v,X_\alpha)=A(v,X)$ for all $v\in ter(X)$;
\item If $X_\alpha$ contains a leafless set $S$, then it contains all edges of $X$ terminating in $V(S)$.

(3) implies that $X_\alpha$ is legal for $c$: 
Let $S$ be any illegal leafless subforest of $X_\alpha$.
Then by (3), any edge of $X$ ending at a vertex in $V(S)$ is already in $X_\alpha$.
Thus $S$ is also an illegal leafless subforest of $X$, and so must be empty.

If $\alpha$ is a limit, we let $X_\alpha=\bigcap_{\beta<\alpha} X_\beta$.
By the induction hypothesis, $X_\alpha$ satisfies (1)-(3). 
It remains to consider the case that $\alpha=\beta+1$ is a successor.
Let $U$ be the union of all leafless subforests of $X_\beta$.
By construction, $U$ itself is a leafless forest. 
If $U$ is empty, we stop the construction and let $X'=X_\beta$.
Otherwise as $X_\beta$ is legal, there is some $st\in U$ such that $A(s,X_\beta\sm U)\geq c(s)+1$.
We obtain $X_\alpha$ from $X_\beta$ by removing all edges 
in $U$ that end at $s$. 
Then (1) is true by construction and (2) is true as $s\notin ter(X)$.
To see (3), let $W\se X_\alpha$ be leafless. Then $W\se U$ and $s\notin V(W)$ by the choice of $X_\alpha$. Thus (3) follows from the induction hypothesis.

This construction terminates as the sets $X_\alpha$ are nested and strictly decrease in every successor step. By construction $X'$ is rayless and satisfies (1) and (2).
\end{enumerate}   
\end{proof}

\begin{proof}[Proof of \autoref{blockage}.]
By \autoref{calc6}, there are sets $X$ and $Y$ such that $\bar r(X)=Y$ and $\bar b(Y)=X$.
Either $A(r,X,Y)\geq 1$ or $A(r,Y,X)\geq 0$. We only consider the case  $A(r,X,Y)\geq 1$, the other case will be analogous.

Our aim is construct a strong blue blockage.
By \autoref{flow}, $A(v, X,Y)\geq 0$ for every $v\in V(X)\sm (B\cup ter(X))$.
By \autoref{legal}, $X$ is legal for $v\mapsto A(v,Y)$. 
So by \autoref{cutoff_trees}, there is some rayless
$X'\se X$ satisfying (1) and (2) from that lemma.
In particular, $A(r,X',Y)\geq 1$ and $A(v,X',Y)\geq 0$ for all $v\in V(X)\sm (B\cup ter(X))$.

By \autoref{rem1} it suffices to show that $\bar r(X')= Y$.
As $X'\se X$ by \autoref{calc1}, $ Y\se \bar r(X')$.
To see that $\bar r(X')\se Y$, we want to apply \autoref{smallest}.
Thus it suffices to show for any 
$st\in E\sm (X'\cup Y)$ that $A(s,X',Y)\geq 0$.
If $s\in V(X)\sm (B\cup ter(X))$, then it has been shown above.
For $s\in B$ this is clear. 
Otherwise $st\notin X\cup Y$ and it follows from $Y=\bar b(X)$ and (2) of \autoref{cutoff_trees}.
Hence $\bar r(X')= Y$.
This completes the proof.
\end{proof}

\section{The special case of the Matroid Intersection Conjecture}\label{sec:int}

First, we introduce the class of matroids for which we prove the Matroid Intersection Conjecture.

\begin{dfn}
A {\em tree $\Tcal$ of matroids} consists of a tree $T$, together with a function $M$ assigning to each node $t$ of $T$ a matroid $M(t)$ on ground set $E(t)$, such that for any two nodes $t$ and $t'$ of $T$, if $E(t) \cap E(t')$ is nonempty then $tt'$ is an edge of $T$.

For any edge $tt'$ of $T$ we set $E(tt') = E(t) \cap E(t')$. We also define the {\em ground set} of $\Tcal$ to be $E = E(\Tcal) = \left(\bigcup_{t \in V(T)} E(t)\right) \setminus \left(\bigcup_{tt' \in E(T)} E(tt')\right)$. 

We shall refer to the edges which appear in some $E(t)$ but not in $E$ as {\em dummy edges} of $M(t)$: thus the set of such dummy edges is $\bigcup_{tt' \in E(T)} E(tt')$.
\end{dfn}

The idea is that the dummy edges are to be used only to give information about how the matroids are to be pasted together, but they will not be present in the final pasted matroid, which will have ground set $E(\Tcal)$. 
We will now consider a type of pasting corresponding to 2-sums. We will make use of some additional information to control the behaviour at infinity: a set $\Psi$ of ends of $T$. 

\begin{dfn}
A tree $\Tcal = (T, M)$ of matroids is {\em of overlap 1} if, for every edge $tt'$ of $T$, $|E(tt')| = 1$. In this case, we denote the unique element of $E(tt')$ by $e(tt')$.

Given a tree of matroids of overlap 1 as above and a set $\Psi$ of ends of $T$, a {\em $\Psi$-pre-circuit} of $\Tcal$ consists of a connected subtree $C$ of $T$ together with a function $o$ assigning to each vertex $t$ of $C$ a circuit of $M(t)$, such that all ends of $C$ are in $\Psi$ and for any vertex $t$ of $C$ and any vertex $t'$ adjacent to $t$ in $T$, $e(tt') \in o(t)$ if and only if $t' \in C$. The set of $\Psi$-pre-circuits is denoted $\overline\Ccal(\Tcal, \Psi)$. 

Any $\Psi$-pre-circuit $(C, o)$ has an {\em underlying set} $\underline{(C, o)} = E \cap \bigcup_{t \in V(C)} o(t)$. Nonempty subsets of $E$ arising in this way are called {\em $\Psi$-circuits} of $\Tcal$. The set of $\Psi$-circuits of $\Tcal$ is denoted $\Ccal(\Tcal, \Psi)$.
\end{dfn}

We shall rely on the following theorem.

\begin{thm}[\cite{BC:determinacy}]
 Let $\Tcal = (T, M)$ be a tree of matroids of overlap 1 and $\Psi$ a Borel set of ends of $T$, then there is a matroid $M_\Psi(T,M)$ whose circuits are the $\Psi$-circuits.
\end{thm}

We can provide the following towards Matroid Intersection.

\begin{thm}\label{case_int}
Let $(T,M)$ and $(T,N)$ be trees of matroids of overlap 1.
Further assume for each node $t$ of $T$ 
that $E(M(t))=E(N(t))$ and 
that both $M(t)$ and $N(t)$ are uniform matroids
whose rank and corank is at least  the degree of $t$.

Then any two matroids
$M_{\Psi_M}(T,M)$ and $M_{\Psi_N}(T,N)$ satisfy the Matroid Intersection Conjecture.
\end{thm}

Let $E(t)$ denote the common ground set of $M(t)$ and $N(t)$ and $E$ the common ground set of
$M_{\Psi_M}(T,M)$ and $M_{\Psi_N}(T,N)$.
A set is \emph{federated} if it
is the union of sets $E(t)\cap E$.
A \emph{packing} for two matroids $K$ and $L$ is a set $P$ together with disjoint spanning sets $S^M$ and $S^N$
of $K\restric_P$ and $L\restric_P$, respectively. 
In a slight abuse of notation, we will refer to the set $P$ itself as a packing.
A \emph{covering} for $K$ and $L$ is a packing for $K^*$ and $L^*$.
We need the following.

\begin{thm}[{\cite[Corollary 3.7]{BC:PC}}]
\label{PC_thm}
$M$ and $N$ satisfy Matroid Intersection if and only if 
$E$ can be partitioned into a packing $P$ for $M$ and $N^*$ and a 
covering $Q$ for $M$ and $N^*$. 
\end{thm}
The following is a consequence of results of \cite{BC:PC}. 

\begin{rem}\label{rem1new}
In order to prove \autoref{case_int}, it suffices 
for each pair $(M_{\Psi_M}(T,M),M_{\Psi_N}(T,N)^*)$ and each $e\in E$
to construct for that pair either a federated packing 
containing $e$ or a 
federated covering containing $e$.
\end{rem}

Before proving \autoref{rem1new}, we prove \autoref{case_int} assuming \autoref{rem1new}.

\begin{proof}[Proof that \autoref{blockage} and \autoref{rem1new} imply \autoref{case_int}.]
For each $t\in V(T)$, we pick an integer $K(t)$. Our way of making these choices will be revealed later.
We obtain $T'$ from $T$ by sticking $|K(t)|$ leaves onto each of its nodes $t$.
Let $B$ consist of those leaves added to nodes $t$ with $K(t)< 0$, and let
$R$ consist of those leaves added to nodes $t$ with $K(t)>0$.
The root $r$ of $T'$ is the unique node with $e\in E(r)$.
Now we apply \autoref{blockage} to $T'$, $B$ and $R$.

First we consider the case that the outcome of this theorem is a red blockage $X_R$.
Let $Z$ be the set of those nodes of $T$ that in $T'$ do not have some $b\in B$ above them.
Let $U$ be the induced subgraph of $T$ whose vertex set is the downclosure of $V(X_R)\cup Z$. Note that $U$ is a tree containing $r$.
Our aim is to show that the union $P$ of the sets $E(t)\cap E$ where $t\in V(U)$ can be given the structure of a packing.
Note that $P$ is federated and contains $e$.

Let $F$ be the set of those dummy edges $e(tt')$ with $tt'\in E(T)\sm E(U)$. 
For $t\in V(U)$, let $R_t$ be the set of those edges in $E(U)\cap X_R$ that end at $t$.
As $X_R$ is a red blockage, $K(t)+|R_t|\geq |F|$.
As the rank of $M(t)$ and the corank of $N^*(t)$ are large enough, $P_t=E(t)\sm R_t\sm F$ is a packing for $(M(t)/R_t\sm F, N(t)^*/R_t\sm F)$ precisely when
$|E(t)|-r(M(t))-r(N^*(t))+|R_t|\geq |F|$.
We now reveal that we have picked $K(t)=|E(t)|-r(M(t))-r(N^*(t))$.
Thus there are  
disjoint spanning sets $S^M_t$ and $S^N_t$ witnessing that $P_t=E(t)\sm R_t\sm F$ is a packing  for the pair above. Moreover, if $uv\in R_v$, we can ensure that both 
$S^M_u$ and $S^N_u$ do not contain $e(uv)$.

We give each node of $U$ one of the colours green and yellow. Which value a particular node gets is revealed later.
If $u$ is green, we ensure that $S^M_u$ contains all dummy edges of $P_u$, which is possible as $r(M(t))$ is at least the degree of $t$.
Similarly, if $u$ is yellow, we ensure that $S^N_u$ contains all dummy edges of $P_u$, which is possible as $r(N^*(t))$ is at least the degree of $t$.
We let $S^M$ and $S^N$ be the set of non-dummy edges of 
$\bigcup_{t\in V(U)} S_t^M$ and $\bigcup_{t\in V(U)} S^N_t$, respectively.

Next we show that $S^M$ spans $P$ in $M_{\Psi_M}(T,M)$.
So let $f\in P\sm S^M$ and let $t_f$ be the unique node with $f\in E(t_f)$.
If $t_f$ is yellow, then $S^M_{t_f}\cap E$ spans $f$.
Moreover, if all neighbours $x$ of $t_f$ are yellow, then $f$ is spanned
by $S^M_{t_f}\cap E$ together with the $S^M_{x}\cap E$ for the neighbours $x$.
This motivates the following definition. We recursively define when a node of $U$ is \emph{good}. All yellow nodes are good. A vertex $v$ is good if all its neighbours are good.
If the unique edge pointing away from $v$ is in $X_R$, then $v$ is already good if all upwards neighbours are good.

As $T$ is locally finite and every decreasing sequence of ordinals is finite, if $t_f$ is good, then $f$ is spanned
by the union of finitely many $S^M_t\cap E$.
Thus, in order to show that $S^M$ spans $P$, it suffices to show that every node of $U$ is good. Now we reveal the colours of the nodes, which are defined recursively.
The root $r$ is green.
The colour of $s\in V(U)-r$ is determined from the colour of the unique node $t$
with $st\in E(T)$: If $st\in X_R$, then $s$ and $t$ get the same colour.
Otherwise $s$ and $t$ get different colours.

Now suppose for a contradiction that there is a node $v$ that is not good.
We pick such a $v$ minimal in $V(U)$. Then $v$ is green
 and by minimality must have an upward neighbour $v'$ that is not good. In particular, $v'$ is green and so $vv'\in X_R$.
For any $st\in X_R$ with $s$ not good, there is a non-good upward neighbour $u$ of $s$. 
In particular, $us\in X_R$. Iterating this argument, we obtain a ray starting in $v$ included in $X_R$. This contradicts that $X_R$ is rayless.
Thus every node of $U$ is good and so $S^M$ spans $P$.
 The proof that 
$S^N$ spans $P$ in $M_{\Psi_N}(T,N)$ is similar.

The case that we get a strong blue blockage is similar since $K(t)=-|E(t)|+r(M^*(t))+r(N(t))$. This completes the proof.
\end{proof}

\begin{proof}[Sketch of the proof of \autoref{rem1new}.]
With the proof of {\cite[Lemma
4.3]{BC:PC}}, it is not hard to show that there is an (inclusion-wise) maximal federated packing $P$ for  $M_{\Psi_M}(T,M)$ and $M_{\Psi_N}(T,N)^*$. 
Let $S$ be the set of those nodes $t$ of $T$ with $E(t)\cap E\se P$.
Let $T'$ be a component of $T\sm S$.
For $t\in V(T')$, we obtain $M'(t)$ from $M(t)$ by contracting all edges of the form $e(st)$ or $e(ts)$ with $s\in S$. We obtain $N'(t)$ from $N(t)$ by deleting all edges of the form $e(st)$ or $e(ts)$ with $s\in S$. 
Let $\Psi_M'$ be the intersection of $\Psi_M$ with the set $\Omega(T')$ of ends of $T'$.
Similarly, $\Psi_N'=\Psi_N\cap \Omega(T')$.
Note that $M(t)$ has at least $r(M(t))$ real edges.
Using this and the corresponding statements for $M^*(t)$, $N(t)$ and $N^*(t)$ and the explicit description of contractions in trees of matroids {\cite[Definition 5.3]{BC:determinacy}}, it is not hard to show that $M_{\Psi_M}(T,M)/P$
is a direct sum of the $M_{\Psi_M'}(T',M')$.
Similarly, $M_{\Psi_N}(T,N)\sm P$
is a direct sum of the $M_{\Psi_N'}(T',N')$.

In particular, $M_{\Psi_M'}(T',M')$ and $M_{\Psi_N'}(T',N')$ are matroids and so satisfy the  assumptions of \autoref{case_int}.
So we may assume that for any $e$ in the common ground set of $M_{\Psi_M}(T,M)/P$
and $M_{\Psi_N}(T,N)^*/ P$ 
there is either a federated packing $P'$ containing $e$ or a 
federated covering $Q'$ containing $e$. By the maximality of $P$,
it can be shown with an argument as in {\cite[Corollary 4.9]{BC:PC}} that there cannot be such a $P'$.
By an argument dual to the one above, we construct a maximal federated covering $Q$ 
for $M_{\Psi_M}(T,M)/P$
and $M_{\Psi_N}(T,N)^*/P$, and we conclude 
that $Q=E\sm P$. Note that $Q$ is a covering for $(M_{\Psi_M}(T,M),M_{\Psi_N}(T,N)^*)$.
Thus \autoref{PC_thm} implies that $M_{\Psi_M}(T,M)$ and $M_{\Psi_N}(T,N)$ satisfy the Matroid Intersection Conjecture.
\end{proof}

\section{Acknowledgement}

I am grateful to Nathan Bowler for useful discussions on this topic.

\bibliographystyle{plain}
\bibliography{literatur}

\begin{thebibliography}{1}

\bibitem{AharoniBerger}
R.~Aharoni and E.~Berger.
\newblock Menger's theorem for infinite graphs.
\newblock {\em Invent.\ math.}, 176:1--62, 2009.

\bibitem{union2}
E.~Aigner-Horev, J.~Carmesin, and J.~Fr{\"o}hlich.
\newblock On the intersection of infinite matroids.
\newblock Preprint (2011), current version available at
  {http://arxiv.org/abs/1111.0606}.

\bibitem{BC:determinacy}
N.~Bowler and J.~Carmesin.
\newblock Infinite matroids and determinacy of games.
\newblock Preprint 2013, current version available at
  {http://arxiv.org/abs/1301.5980}.

\bibitem{BC:PC}
N.~Bowler and J.~Carmesin.
\newblock Matroid intersection, base packing and base covering for infinite
  matroids.
\newblock To appear in Combinatoica, available at
  {http://arxiv.org/pdf/1202.3409} ‎.

\bibitem{matroid_axioms}
H.~Bruhn, R.~Diestel, M.~Kriesell, R.~Pendavingh, and P.~Wollan.
\newblock Axioms for infinite matroids.
\newblock Advances in Mathematics, to appear; arXiv:1003.3919 [math.CO].

\bibitem{DiestelBook10}
R.~Diestel.
\newblock {\em Graph {T}heory \emph{(4th edition)}}.
\newblock Springer-Verlag, 2010.
\newblock \\ Electronic edition available at:\\ {\small\tt
  http://diestel-graph-theory.com/index.html}.

\bibitem{Oxley}
J.~Oxley.
\newblock {\em {M}atroid {T}heory}.
\newblock Oxford University Press, 1992.

\end{thebibliography}

\end{document}